\documentclass[oneside,10pt,reqno]{amsart}
\usepackage{amssymb,amsmath,amsthm,bbm,enumerate,hyperref,amsthm}
 \usepackage{amsaddr}
\usepackage[semicolon]{natbib}
\usepackage[shortlabels]{enumitem}
\usepackage[T1]{fontenc}

\theoremstyle{definition}
\newtheorem{definition}{Definition}
\theoremstyle{theorem}
\newtheorem{proposition}[definition]{Proposition}
\newtheorem{lemma}[definition]{Lemma}

\newtheorem{corollary}[definition]{Corollary}

\numberwithin{equation}{section}
\numberwithin{definition}{section}
\theoremstyle{remark}
\newtheorem{remark}[definition]{Remark}
\newtheorem{example}[definition]{Example}

\def\PP{\mathsf{P}}
\def\QQ{\mathsf{Q}}
\def\EE{\mathsf{E}}
\def\FF{\mathcal{F}}
\def\TT{\mathcal{T}}
\def\RR{\mathcal{R}}
\def\AA{\mathcal{A}}
\def\BB{\mathcal{B}}
\def\supp{\mathrm{supp}}
\def\tT{\tilde{T}}
\def\oF{\overline{F}}
\def\id{\mathrm{id}}
\def\Exp{\mathrm{Exp}}

\def\dd{\mathrm{d}}
\def\stg{\geq_{\mathrm{st}}}
\def\stl{\leq_{\mathrm{st}}}

\begin{document}
\title[Laws universally ordered under stochastic restart]{On laws exhibiting universal ordering under stochastic restart}
\begin{abstract}
For each of (i) arbitrary stochastic reset, (ii) deterministic reset with arbitrary period, (iii) reset at arbitrary constant rate, and then in the sense of either (a) first-order stochastic dominance or (b) expectation (i.e. for each of the six possible combinations of the preceding), those laws of random times are precisely characterized that are  rendered no bigger [rendered no smaller; left invariant] by all possible restart laws (within the classes (i), (ii), (iii), as the case may be). Partial results in the same vein  for reset with branching are obtained. In particular it is found that deterministic and arbitrary stochastic restart lead to the same characterizations,  but this equivalence fails to persist for exponential (constant-rate) reset. 
\end{abstract}
\author{Matija Vidmar}
\address{Department of Mathematics, University of Ljubljana and Institute of Mathematics, Physics and Mechanics, Slovenia}
\email{matija.vidmar@fmf.uni-lj.si}
\thanks{Financial support from the Slovenian Research Agency is acknowledged (programme No. P1-0402). I thank two anonymous Referees whose comments and suggestions have helped to improve the presentation of this paper.
}
\maketitle
\section{Introduction}
\subsection{Problem delineation}\label{subsection:problem}
Let $\TT$ and $\RR$ be two given probability laws on the Borel  subsets of $[0,\infty]$. Suppose some time-to-completion (resp. time-to-failure) of a process has, ex-ante,  law $\TT$. Imagine further that --- with a view to minimizing (resp. maximizing) this time --- rather than letting the process just run its course, we reset the process at a (random, independent) time distributed according to the law $\RR$, starting an independent copy of the process thereafter, and that then we \emph{repeat this over and over again} until completion (resp. failure).  The time to the latter now has a different, ex-post, law, $\TT^\RR$, say. In (almost) precise terms: if $(T_i)_{i\in \mathbb{N}}$ is a sequence of independent  with law $\TT$ distributed random variables, and $(R_k)_{k\in \mathbb{N}}$ is an independent sequence of independent with law $\RR$ distributed random variables, then $\TT^\RR$ is the law of the random time $\tT$ given by:
 \begin{equation}\label{eq:reset}
 \tT=R_1+\cdots+R_{k-1}+T_k\text{ on }\{R_1<T_1,\ldots,R_{k-1}<T_{k-1},T_k\leq R_k\},\quad k\in \mathbb{N}
 \end{equation} (a.s., under certain assumptions on $\RR$ and $\TT$; we make it fully precise below).  In less formal paralance: the distribution of  $\tT$ solves the  equation (we \emph{write} it with random variables, but we \emph{mean} it in distribution):
 \begin{equation}\label{eq:reset-informal}
 \tT=
 \begin{cases}
 T & \text{ if $T\leq R$}\\
 R+\tT'& \text{ if $R<T$}
 \end{cases},
 \end{equation}
 where $R$, $T$, $\tT'$ are independent, $R$ has law $\RR$, $T$ has law $\TT$, and the law of $\tT'$ is that of $\tT$. Of special interest are the cases when $\RR$ corresponds to: (i) resetting the process at deterministic epochs that are a fixed amount of time $r\in (0,\infty)$ apart (viz. $\RR=\delta_r$, the Dirac mass at $r$); (ii) restarting the process at a constant rate $\lambda\in (0,\infty)$ (viz. $\RR=\Exp(\lambda)$, the exponential law of mean $\lambda^{-1}$). 
 
Now, depending on the particulars of $\RR$ and $\TT$, $\TT^\RR$ may, or may not be smaller (resp. bigger) than $\TT$. Here smaller (resp. bigger) is meant in either of the two natural senses: in first-order stochastic dominance;  in mean. A natural qualitative question arises: \emph{which are the laws $\TT$ that are always rendered (a) no bigger (resp. (b) no smaller), or that are  left (c) invariant --- in first-order stochastic dominance or just in mean, as the case may be --- under restart (arbitrary, deterministic, of constant rate)}? Indeed if a distribution falls under case (a) (resp. (b)), then (and only then) we are assured that restart can do no harm when seeking to minimize (resp. maximize) the time to completion (resp. failure), and can further even be strictly beneficial if in addition (c) is precluded. Given the  relevance of stochastic reset in applied probability (see Subsection~\ref{subsection:connections} below) it certainly  seems worthwhile to record the precise conditions under which this occurs (as well as the various (non)implications between them). 



By way of example, let $\TT$ be the law of the first passage time above the level $1$ of a standard linear Brownian motion $B$; it is finite a.s. with infinite mean. Restarting $B$  at fixed deterministic epochs, say,  will render this first hitting/passage time to be of finite mean. But will it be even first-order stochastic dominated by the time without restart? And if we restart an exponentially distributed time, then its law is surely left invariant by this resetting. Is it the only law of which this is true? The puporse of this paper is to provide precise answers to such and similar questions. 

In case when we are seeking to minimize the time to completion, a natural complement to the above consists in allowing the effort to be increased by the same factor $l\in \mathbb{N}_{\geq 2}$ on each iteration, so that on first reset (if it occurs) not one, but $l$ independent processes with the same distribution as the original one are started,  termination  occuring at the minimum of their respective times to completion, unless a second reset occurs before that, in which case $l^2$ independent procecess are started, and so on.  We will have occasion to give (partial) answers to the above questions also in the context of this ``restart with $l$-fold branching''. 

\subsection{A flavor of the results}\label{subsection:flavor}
A very inexhaustive, completely partial, and slightly informal indication of the results to follow is given by the following table. In it, for a law $\TT$ with tail function $\oF$, it is indicated what the condition is that renders $\TT$ always to have a certain behaviour relative to first-order stochastic dominance under a given class of reset laws $\RR$. For instance, the two entries under ``no bigger'' \& ``every exponential reset'' (one entry for the case when there is no branching and one entry for when there is  $l$-fold branching with $l\in \mathbb{N}_{\geq 2}$), give the condition that the law $\TT$ is no bigger in the sense of first-order stochastic dominance under every exponential reset. 
The conditions must be read ``universally quantified'', e.g. 
``$\frac{1}{t}\int_0^t\oF(u)\oF(t-u)\dd u\leq \oF(t)$'' should really come equipped with  ``for all $t\in (0,\infty)$''.
\noindent \footnotesize
\begin{tabular}{|p{1.3cm}||p{2.8cm}|p{3.6cm}||p{2.8cm}|p{3.8cm}|}\hline
 & \multicolumn{2}{c||}{no branching} & \multicolumn{2}{c|}{$l$-fold branching}\\\hline
under /\phantom{} the law is& every [or just every deterministic] reset & every exponential reset & every [or just every deterministic] reset & every exponential reset\\\hline
no bigger & $\overline{F}(x)\overline{F}(y)\leq \overline{F}(x+y)$ & $\frac{1}{t}\int_0^t\oF(u)\oF(t-u)\dd u\leq \oF(t)$ &  $\overline{F}(x)\overline{F}(y)^l\leq \overline{F}(x+y)$ & $\frac{1}{t}\int_0^t\oF(u)\oF(t-u)^l\dd u\leq \oF(t)$\\\hline
invariant & \multicolumn{2}{c||}{exponential distributions} & \multicolumn{1}{c|}{cannot happen}  & \multicolumn{1}{c|}{?}\\\hline
\end{tabular}
\normalsize

The ``no smaller'' case is analagous (when relevant). For the behaviour relative to the order in mean,  for more in-depth statements and for (counter)examples the reader must consult the main body of the text.

\subsection{Connections to existing literature}\label{subsection:connections}
The author was orginally motivated to explore the above subject matter after reading the paper \citep{pal}. Indeed it is therein that the procedure described in the last paragraph of Subsection~\ref{subsection:problem} was called restart with branching (following on from \citep{Eliazar}, where the term ``branching search'' was introduced). 

Now, in a main result of \citep{pal}, a universal condition is found on $\TT$ under which exponential restart  with branching \emph{of sufficiently small rate} renders the completion time to have a smaller mean (than without restart), generalizing the analogous no-branching result of \citep{reuveni}. On the other hand, in \citep{Eliazar} the analysis is made with a \emph{fixed deterministic} reset law $\RR=\delta_\tau$ (where $\tau\in (0,\infty)$); in particular those laws are characterized that are rendered invariant by $\delta_\tau$-restart with some $\beta$-fold branching (where  $\beta\in \mathbb{N}$ too is fixed). Further, in \cite{checkin} it is found  that the only cases when \emph{some} restart can do harm in the mean (in the sense of  increasing the expected time-to-completion) correspond to  (in slightly informal terms)  ``the single run hitting probability densities decaying faster than exponentially''. And finally, in the context of a ballistic random walker with a random velocity, that is reset to its starting position at random epochs, \cite{Villarroel} observe that the mean time to reaching a given level by the walker ``has exponential distribution regardless of the distribution of reset times'' provided the distribution of the velocity is inverse-exponential, and further that ``the richness of the system
manifests in surprising behaviors: in certain cases the hitting time may be in inverse
relation to the reset activity.'' 

The questions raised in Subsection~\ref{subsection:problem} (and their answers) are then a complement to what has hitherto been considered (established): we are asking for the precise conditions that characterize (weak) improvement/invariance over arbitrary reset (possibly only within the deterministic, or exponential class), separately for the case without branching, and with a given $l$-fold branching, $l\in \mathbb{N}_{\geq 2}$; moreover, we do so both relative to first-order stochastic dominance as well as relative to dominance in expectation (where the physics literature tends to consider only dominance in mean).

Looking out further in terms of existing literature, stochastic restart (a.k.a. reset/restart search) has been the subject of relatively intense recent study, especially in the statistical physics literature, and has been so in problems ranging across biology, chemistry, physics and data/computer science. 
We refer the reader to the introductions of the papers \citep{reuveni,pal-reuveni,Eliazar,checkin,Villarroel,unified,pal} for an overview of this literature that looks at reset from the point of view of minimizing the time to completion. From the opposite stance --- maximizing the time to failure --- restart falls naturally under reliability (ageing, longevity) theory \citep{desphande,weiss,barlow}, in which case it is better called (preventative) replacement.

\subsection{Article organization} 
The main analysis is carried out in Sections~\ref{section:restart} and~\ref{section:branching} dealing with reset without and with branching, respectively.  Section~\ref{section:illustration} closes with a simple illustration, but many  (counter)examples are already given along the way. 


\subsection{General notation}  Given a probability law $\QQ$ we will  write $\QQ[W]$ for $\EE_\QQ[W]$, $\QQ[W;A]$ for $\EE_\QQ[W\mathbbm{1}_A]$ and $\QQ[W\vert A]$ for $\EE_\QQ[W\vert A]$. If a probability measure is denoted by $\PP$, then its expectation operator is denoted simply with the symbol $\EE$.    We use $\supp(\TT)$ to denote the support of a probability law $\TT$ on $[0,\infty]$; $\id$ is the identity on $[0,\infty]$. Given an expression $f(x)$ defined for $x\in X$ we write sometimes $f(\cdot)$ to mean the function $(X\ni x\mapsto f(x))$.

\subsection{A vademecum}\label{vademecum} For those that prefer working with probability density functions rather than probability laws,  if a real random variable $X$ has law $\mathcal{X}$ under the probability $\PP$, and if $\mathcal{X}$ is absolutely continuous with density (p.d.f.) $f_X$, then $\EE[g(X)]=\mathcal{X}[g]=\int_\mathbb{R} g(x)f_X(x)\dd x$ for measurable $g:\mathbb{R}\to \mathbb{R}$, in particular $\EE[\mathbbm{1}_A]=\PP(X\in A)=\mathcal{X}(A)=\mathcal{X}[\mathbbm{1}_A]=\int_Af_X(x)\dd x$ for Borel $A\subset\mathbb{R}$. For instance, if $X$ is $[0,\infty)$-valued, then $\EE[X]=\EE[\id(X)]=\mathcal{X}[\id]=\int_0^\infty x f_X(x)\dd x$ is the mean of $X$,  $\EE[X-a\vert X>a]=\mathcal{X}[\id-a\vert (a,\infty)]=\int_a^\infty (x-a)f_X(x)\dd x/\int_a^\infty f_X(x)\dd x$ is the residual mean of $X$ at some given $a\in (0,\infty)$ (assuming $\PP(X>a)>0$), etc. Using this ``dictionary'', the results below may be transcribed in terms of p.d.f. if one pleases (of course this means one must assume the existence of densities, viz. absolute continuity).

\section{Reset and stochastic order}\label{section:restart}

Let $\TT$ be a probability law on the Borel sets of $[0,\infty]$, viz. the probability law of a random time. To avoid trivial considerations we will assume that $\TT((0,\infty])>0$ and that  $\inf \supp (\TT)=0$, i.e. if $T\sim \TT$, then $T$ is  positive with a positive probability, and smaller than any given positive number with a positive probability. We denote by $F:[0,\infty)\to [0,1]$ the distribution function of $\TT$ and let $\oF:=1-F$ be the tail function of $\TT$. Note $\oF$ determines $\TT$ uniquely and that in terms of $\oF$ our standing assumptions read: $\oF(0)>0$ and $\oF(r)<1$ for all $r\in (0,\infty)$. We set $F(\infty):=1$ and $\oF(\infty):=0$ (a convention that does not preclude $\lim_{u\to \infty} \oF(u)>0$).

\begin{definition}[Reset law]\label{definition:reset-order}
A probability law $\RR$ on the Borel sets of $[0,\infty]$, with $\RR((0,\infty])\land \RR([0,\infty))>0$ is called a reset law.
\end{definition}
\begin{remark}
For each $r\in (0,\infty)$, $\delta_r$ and $\Exp(r)$ are reset laws. 
\end{remark}
\begin{remark}
Under our assumptions any reset law $\RR$ satisfies $\TT\times \RR(\{(t,r)\in [0,\infty]^2:t\leq r\})>0$, a condition that is crucial to the definition of reset to make sense, cf. \eqref{eq:reset}. Indeed the only way for $\RR=\delta_0$ to satisfy $\TT\times \RR(\{(t,r)\in [0,\infty]^2:t\leq r\})>0$ is for $\TT(\{0\})>0$, in which case $\TT^\RR=\delta_0$. Hence when identifying conditions under which $\TT$ is rendered no bigger (resp. no smaller) by reset we lose no generality by insisting on the reset laws to satisfy ``$\RR((0,\infty])>0$'', since $\delta_0$ is no bigger than any law (resp., as it will emerge, if $\TT(\{0\})>0$, then $\TT$ cannot be no smaller under arbitrary (or just deterministic) reset anyway, the assumption ``$\RR((0,\infty])>0$'' on the reset laws notwithstanding).  On the other hand, the assumption $\RR([0,\infty))>0$ simply excludes the trivial law $\RR=\delta_\infty$ for which $\TT^\RR=\TT$. 
\end{remark}
\begin{definition}[Reset]\label{definition:law-reset-by}
Let $\RR$ be a reset law. We define two new probability laws, $\TT^\RR$ and $\TT_\RR$, on the Borel sets of $[0,\infty]$ as follows. First, let $(T_k)_{k\in \mathbb{N}}$ be a sequence of $[0,\infty]$-valued i.i.-with law $\TT$-d. random variables and let $(R_k)_{k\in \mathbb{N}}$ be an independent sequence of $[0,\infty]$-valued i.i.-with law $\RR$-d. random variables, all defined on a commom probability space $(\Omega,\FF,\PP)$. Then $\TT^\RR$ is the law of the random time $\tilde{T}:\Omega\to [0,\infty]$ specified a.s. (because $\PP(T_1\leq R_1)>0$) by $$\tT=R_1+\cdots+R_{k-1}+T_k\text{ on }\{R_1<T_1,\ldots,R_{k-1}<T_{k-1},T_k\leq R_k\},\quad k\in \mathbb{N}.$$
We denote by $\oF^\RR$ the tail function of $\TT^\RR$. To define $\TT_\RR$ let, again on a common probability space $(\Omega,\FF,\PP)$, $R$ be a $[0,\infty]$-valued random time with law $\RR$ independent of the $[0,\infty]^2$-valued pair $(T_1,T_2)$ that has law $\TT\times \TT$. Then $\TT_\RR$ is the law of the random time $T':=T_1\mathbbm{1}_{\{T_1\leq R\}}+(R+T_2)\mathbbm{1}_{\{R<T_1\}}$. 
\end{definition}
\emph{It will not be formally relevant, and we do not assume it, but we mean to investigate below the situation when $\TT$ is non-lattice. In the lattice case the natural considerations would be different (only the reset laws carried by the same lattice as the one which supports $\TT$ would be meaningfully included in the discussion).} 

\subsection{Reset and the usual stochastic order}
We focus first on the behavior of $\TT$ under reset relative to first-order stochastic dominance. To avoid any ambiguity, we recall that given two probability laws $\AA$ and $\BB$ on the Borel sets of $[0,\infty]$, then $\BB$ is said to first-order stochastic dominate $\AA$, written $\BB\stg\AA$ (or that $\AA$ is first-order stochastic dominated by $\BB$, written $\AA\stl\BB$) iff $\BB[\mathbbm{1}_{(t,\infty]}]=\BB((t,\infty])\geq \AA((t,\infty])=\AA[\mathbbm{1}_{(t,\infty]}]$ for all $t\in [0,\infty)$, which  then implies that $\BB[G]\geq \AA[G]$ for all nondecreasing $G:[0,\infty]\to [0,\infty]$  \citep[Paragraph 1.A.1]{reset}.


Our first proposition identifies precisely the laws that are, in the sense of first-order stochastic dominance, no bigger (no smaller, invariant) under stochastic reset, equivalently (as it emerges) deterministic reset. It becomes almost obvious once stated; still some minimal amount of care is needed in the proof to handle the general resetting.  We adhere to  the convention $\log 0:=-\infty$ below.

\begin{proposition}\label{proposition}
The following statements are equivalent.
\begin{enumerate}[(i)]
\item\label{thm:i} $\TT$ is no bigger (resp. no smaller, invariant) under reset:
 $\TT^{\RR}\stl\TT$ (resp. $\TT^{\RR}\stg\TT$, $\TT^\RR=\TT$) for all reset laws $\RR$. 
\item\label{thm:ii} $\TT$ is no bigger (resp. no smaller, invariant) under deterministic reset: 
$\TT^{\delta_r}\stl\TT$ (resp. $\TT^{\delta_r}\stg\TT$, $\TT^{\delta_r}=\TT$) for all $r\in (0,\infty)$.
\item\label{thm:iii}  $-\log \oF:[0,\infty)\to [0,\infty]$ is subadditive (resp. superadditive, additive); i.e.  for all $r\in [0,\infty)$ and $t\in [0,\infty)$,
\[\TT[\mathbbm{1}_{(t,\infty]}(\cdot -r)\vert (r,\infty]]\text{ is }\geq \TT[\mathbbm{1}_{(t,\infty]}]\text{ (resp. $\leq \TT[\mathbbm{1}_{(t,\infty]}]$, $=\TT[\mathbbm{1}_{(t,\infty]}]$)}
\]
with  $\TT((r,\infty])>0$ implicit (resp. whenever  $\TT((r,\infty])>0$, with  $\TT((r,\infty])>0$ implicit); in still other words, $\oF$ is supermultiplicative (resp. submultiplicative, multiplicative).  
\item\label{thm:v}  $\TT_\RR\stl\TT$ (resp. $\TT_\RR\stg\TT$, $\TT_\RR=\TT$) for all reset laws $\RR$. 
 \item\label{thm:iv}  $\TT_{\delta_r}\stl\TT$ (resp. $\TT_{\delta_r}\stg\TT$, $\TT_{\delta_r}=\TT$) for all $r\in (0,\infty)$. 
 \end{enumerate}
 Furthermore, $\oF$ is multiplicative iff  $\TT=\Exp(\lambda)$ for some $\lambda\in (0,\infty)$.
\end{proposition}
Before turning to the proof we make some observations concerning the above.
\begin{remark}
If $\oF(0)=1$ and $\oF$ is supermultiplicative, then $\TT$ can have no finite atoms. More generally, if $\oF$ is supermultiplicative, then all finite atoms of $\TT$ have, relative to the left limit, size at most $1-\oF(0)$.
\end{remark}
\begin{remark}
In the preceding proposition, we cannot ``separate  condition \ref{thm:iii} according to $t\in [0,\infty)$'': for a fixed $t\in [0,\infty)$, the implication 
\begin{quote}
$\TT^{\delta_r}[\mathbbm{1}_{(t,\infty)}]\leq \TT[\mathbbm{1}_{(t,\infty)}]$ for all $r\in (0,\infty)$ $\Longleftarrow$  $\TT((r,\infty])>0$ and $\TT[\mathbbm{1}_{(t,\infty]}(\cdot -r)\vert (r,\infty]]\geq \TT[\mathbbm{1}_{(t,\infty]}]$ for all $r\in (0,\infty)$
\end{quote}
fails to hold in general, as the next example demonstrates. We will however see  that if we replace in the preceding $\mathbbm{1}_{(t,\infty)}$ with the identity map, then the implication ``$\Longleftarrow$'' becomes an equivalence ``$\Longleftrightarrow$'' (also if in the first statement we insist on the inequality for all reset laws $\RR$ in lieu of $\delta_r$).  But this then cannot be a (completely) ``universal'' phenomenon (it must presumably owe itself to the particularities of the identity map /additivity/).
\end{remark}
\begin{example}
 Suppose $\oF(t)=\mathbbm{1}_{[0,1)\cup (2,\infty)}(t)e^{- t}+e^{-2}\mathbbm{1}_{[1,2]}$ for $t\in [0,\infty)$. We have $\oF(u)\oF(1)\leq \oF(1+u)$ for all $u\in [0,\infty)$ (but $\oF(1-a)\oF(a)>\oF(a+(1-a))$ for all $a\in (0,1)$). Set also $\RR=\delta_{\frac{1}{2}+\epsilon}$ for an $\epsilon\in (0,\frac{1}{2})$. We have on the one hand $\TT((1,\infty])=e^{-2}$. On the other hand $\TT^\RR((1,\infty])\geq  \TT((\frac{1}{2}+\epsilon,\infty])\TT((\frac{1}{2},\infty])=e^{-1-\epsilon}$. Hence $\TT^\RR((1,\infty])>\TT((1,\infty])$.
 \end{example}
 \begin{remark}
 If $-\log \oF$ is subadditive, then from the right-continuity of $\oF$ and from $\oF(0)=\TT((0,\infty])>0$, it follows that $\oF>0$ (everywhere). 
 On the other hand, if $-\log\oF$ is superadditive, then we must have $\oF(0)=1$. 
 \end{remark}
 \begin{remark}
 Property \ref{thm:iii} is used as the defining property of the class of ``new worse than used'' (resp. ``new better than used'', ``new same as used'') distributions in reliability theory \citep{nofal,rao}. See \citep[Lemma~2.1]{el} for another (unrelated) statement involving stochastic order that is equivalent to the submultiplicativity of $\oF$.
 \end{remark}
 \begin{corollary}\label{corollary:finiteness-of-means}
If $\TT$ is no bigger (resp. no smaller) under [deterministic] reset, then for all  nondecreasing $G:[0,\infty]\to [0,\infty]$ and for all reset laws $\RR$, the finiteness (resp. divergence) of $\TT[G]$, i.e. of the $G$-moment for $\TT$, implies the finiteness (resp. divergence) of $\TT^\RR[G]$, i.e. of the $G$-moment for $\TT^\RR$. \qed
\end{corollary}
\begin{remark}\label{remark:moments}
The latter corollary may be seen as a complement to the observation of \cite{checkin} that ``the number of moments of the hitting time
distribution under resetting is not less than the sum of the numbers of moments of the resetting time distribution and the hitting time distribution without resetting''.
\end{remark}
\begin{proof}[Proof Proposition~\ref{proposition}]
Suppose first that $\TT$ is no bigger under deterministic reset. Let $r\in (0,\infty)$ and set $\RR:=\delta_r$. Assume the setting and notation of Definition~\ref{definition:law-reset-by}. Let also  $k\in \mathbb{N}_0$ and $t\in [kr,(k+1)r)$. Then
\begin{align}
\oF^{\delta_r}(t)
&=\PP(r<T_1,\ldots,r<T_{k}, t-rk<T_{k+1}\leq r)+\PP(r<T_1,\ldots,r<T_{k+1})\nonumber\\
&=\oF(r)^k(\oF(t-rk)-\oF(r))+\oF(r)^{k+1}
=\oF(r)^k\oF(t-rk)\leq \oF(t).\label{eq:deterministic}
\end{align}
Now let $\{x,y\}\subset [0,\infty)$, $0<x> y$. Setting $r:=x$, $k:=1$, $t:=x+y$ in the preceding shows that $\oF(x+y)\geq \oF(x)\oF(y)$. Taking limits, exploiting the right-continuity of $\oF$, ceteris paribus, the premise $0<x>y$ can be weakened to $0\leq x\geq y$ and the subaddivity of $-\log \oF$ follows. Similarly if we had assumed that $\TT$ was no smaller (invariant) under deterministic reset, the superaddivity (additivity) of $-\log\oF$ would have obtained. 

Now suppose that $-\log \oF$ is subadditive. Let $\RR$ be a reset law and assume again the setting and notation of Definition~\ref{definition:law-reset-by}. Let also $t\in [0,\infty)$. We compute and estimate:
\begin{align}
\oF^\RR(t)&=\sum_{k=1}^\infty \PP(R_1<T_1,\ldots,R_{k-1}<T_{k-1},T_k\leq R_k,R_1+\cdots+R_{k-1}+T_k>t)\nonumber\\
&= \sum_{k=1}^\infty \EE\big[\oF(R_1)\cdots \oF(R_{k-1})(\oF(t-R_1-\cdots-R_{k-1})\mathbbm{1}_{\{R_1+\cdots +R_{k-1}\leq t\}}\nonumber\\
&\hspace{5cm}+\mathbbm{1}_{\{R_1+\cdots+ R_{k-1}> t\}}-\oF(R_k));R_1+\cdots +R_k> t\big]\nonumber\\
&\leq \sum_{k=1}^\infty \EE\big[\oF(t)\mathbbm{1}_{\{R_1+\cdots +R_{k-1}\leq t\}}+\oF(R_1)\cdots \oF(R_{k-1})\mathbbm{1}_{\{R_1+\cdots+ R_{k-1}> t\}}\nonumber\\
&\hspace{3cm}-\oF(R_1)\cdots \oF(R_{k-1})\oF(R_k);R_1+\cdots +R_k> t\big]\nonumber\\
&=\sum_{k=1}^\infty \EE\big[\oF(t)\mathbbm{1}_{\{R_1+\cdots +R_{k-1}\leq t<R_1+\cdots+R_k\}}]+\sum_{k=1}^\infty\EE[\oF(R_1)\cdots \oF(R_{k-1})\mathbbm{1}_{\{R_1+\cdots+ R_{k-1}> t\}}\nonumber\\
&\hspace{3cm}-\oF(R_1)\cdots \oF(R_k)\mathbbm{1}_{\{R_1+\cdots +R_k> t\}}]=\oF(t),\label{eq:stochastic}
\end{align}
where the final equality follows from the fact that by the law of large numbers, a.s. $R_1+\cdots +R_k\uparrow \infty$ as $k\to \infty$ (because $\PP(R_1>0)>0$) and (to handle the telescopic series) from the fact that $\EE[\oF(R_1)\cdots \oF(R_{k});R_1+\cdots+ R_{k}> t]\leq \EE[\oF(R_1)\cdots \oF(R_{k})]=\PP(R_1<T_1)^k\downarrow 0$ as $k\to \infty$ (because $\PP(R_1<T_1)<1$). Similarly submultiplicativity (multiplicativty) of $\oF$ would yield that $\TT$ is no smaller (invariant) under reset. 

Next, assume \ref{thm:iv} holds. We may express, for $r\in (0,\infty)$, 
\begin{equation}\label{eq:single-reset-det}
\TT_{\delta_r}((t,\infty])=\oF(t)\mathbbm{1}_{(t,\infty)}(r)+\oF(r)\oF(t-r)\mathbbm{1}_{[0,t]}(r),\quad t\in [0,\infty),
\end{equation}
whence \ref{thm:iii} follows. Conversely assume the latter and let $\RR$ be a reset law. We see that
\begin{equation}
\TT_\RR((t,\infty])
=\int \oF(t)\mathbbm{1}_{(t,\infty]}(r)+\oF(r)\oF(t-r)\mathbbm{1}_{[0,t]}(r)\RR(\dd r),\quad t\in [0,\infty).\label{eq:single-reset}
\end{equation}
The assertion of \ref{thm:v} follows.

Assume now $\oF(x)\oF(y)=\oF(x+y)$ whenever $\{x,y\}\subset [0,\infty)$.  We see that necessarily, for all $x\in (0,\infty)$, $\oF(x)>0$ , for otherwise the functional equation for $\oF$ would imply that $\oF(x/2^n)=0$ for all $n\in \mathbb{N}$, which would contradict the right-continuity of $\oF$ at $0$ (and the fact that $\oF(0)>0$). Taking logarithms the theory of Cauchy's functional equation implies that $\oF(x)=e^{-\lambda x}$ for all $x\in [0,\infty)$, for some $\lambda\in \mathbb{R}$; necessarily then $\lambda\in (0,\infty)$ (because $\TT([0,\infty))>0$). It means that $\TT=\Exp(\lambda)$. 
\end{proof}

\begin{remark}
It is clear from \eqref{eq:stochastic} that if $\oF$ is supermultiplicative (resp. submultiplicative), then even if in Definition~\ref{definition:law-reset-by} we, ceteris paribus, drop the premise that the $R_i$, $i\in \mathbb{N}$, are identically distributed with law $\RR$, and ask instead merely that $\prod_{k=1}^\infty \PP(R_k<T_k)=0$  (resp. and $\sum_{l=1}^k R_l\uparrow \infty$ a.s. as $k\to\infty$), then still the law of $\tT$ is first-order stochastic dominated by (resp. first-order stochastic dominates) $\TT$.
\end{remark}


We now look at exponential reset. Perhaps somewhat surprisingly the condition for $\TT$ to be no bigger/no smaller in first-order stochastic dominance is now different, but invariance is still characteristic of exponential laws.

\begin{proposition}\label{proposition:exponential}
The following statements are equivalent. 
\begin{enumerate}[(i)]
\item $\TT$ is no bigger (resp. no smaller, invariant) under exponential reset: $\TT^{\Exp(\mu)}\stl\TT$ (resp. $\TT^{\Exp(\mu)}\stg \TT$, $\TT^{\Exp(\mu)}= \TT$) for all $\mu\in (0,\infty)$. 
\item $\oF^{\Exp(\mu)}(t)$ is $\leq \oF(t)$ (resp. $\geq \oF(t)$, $=\oF(t)$) for all sufficiently small $\mu\in (0,\infty)$ for all $t\in [0,\infty)$ (in this order of qualification!).
\item $\frac{1}{t}\int_0^t\oF(u)\oF(t-u)\dd u$ is $\leq \oF(t)$ (resp. $\geq \oF(t)$, $=\oF(t)$) for all $t\in (0,\infty)$.
\end{enumerate}
Furthermore, $\frac{1}{t}\int_0^t\oF(u)\oF(t-u)\dd u=\oF(t)$ for all $t\in (0,\infty)$ iff $\TT=\Exp(\lambda)$ for some $\lambda\in (0,\infty)$.  
\end{proposition}
\begin{example}
Let $\oF=\frac{1}{2}\mathbbm{1}_{[0,1)}+\frac{1}{4}\mathbbm{1}_{[1,\frac{3}{2})}+\frac{1}{6}\mathbbm{1}_{[\frac{3}{2},\infty)}$. Using the above equivalent conditions it is then elementary (if tedious) to check that $\TT$ is no bigger under exponential reset, but it is not no bigger under reset (a little thought reveals that the fact that $\TT$ has atoms is not crucial, it just simplifies the computations). 
\end{example}
\begin{remark}
We see that in order for $\TT$ to be no smaller under exponential reset it is necessary that $\oF(0)=1$.
\end{remark}
\begin{proof}
We leave the ``(resp. /\ldots/)'' parts to the reader. Let $\mu\in (0,\infty)$, set $\RR:=\Exp(\mu)$, and assume the setting and notation of Definition~\ref{definition:law-reset-by}. Remark that $\oF(0)=\oF^\RR(0)$. Define $S_n:=R_1+\cdots+R_n$ for $n\in \mathbb{N}_0$; $S_n$, $n\in \mathbb{N}$, are the arrival times of a homogenous Poisson process $N=(N_u)_{u\in [0,\infty)}$ of intensity $\mu$. We then see from \eqref{eq:stochastic} that, for any given $t\in (0,\infty)$, 
\begin{align}
\PP(\tT>t)&=\sum_{k=0}^\infty \EE\big[\oF(R_1)\cdots\oF(R_k)\oF(t-S_k);N_t=k]\nonumber\\
&=\sum_{k=0}^\infty \EE\big[\oF(R_1)\cdots\oF(R_k)\oF(t-S_k)\vert N_t=k]\frac{(\mu t)^k}{k!}e^{-\mu t}\label{exp-reset}
\end{align}
If now this is $\leq \oF(t)$ for all sufficiently small $\mu\in (0,\infty)$, then subtracting $\oF(t)e^{-\mu t}$ in this inequality, dividing by $\mu t$ and letting $\mu\downarrow 0$, we find that $\EE[\oF(R_1)\oF(t-R_1)\vert N_t=1]\leq \oF(t)$. Since, conditionally on $\{N_t=1\}$, $R_1$ is uniform on $[0,t]$, we conclude that $\frac{1}{t}\int_0^t\oF(u)\oF(t-u)du\leq \oF(t)$. Conversely, suppose the preceding inequality obtains for all $t\in (0,\infty)$. Let $t\in (0,\infty)$, $k\in \mathbb{N}_{\geq 2}$, and let, on some probability space $(\Omega,\FF,\PP)$,  $U_1,\ldots,U_k$ be independent uniformly on $[0,t]$ distributed  random variables. Let $U_{(0)}^k,U_{(1)}^k,\ldots,U_{(k)}^k,U^{k}_{(k+1)}$ be the order statistics of $0,U_1,\ldots,U_k,t$, and set $V_n^k:=U_{(n)}^k-U_{(n-1)}^k$ for $n\in \{1,\ldots,k+1\}$. In words, $V_1^k,\ldots,V_{k+1}^k$ are the ``spaces'' between the random variables $0,U_1,\ldots,U_k,t$. Analogous quantitities are introduced for $0,U_1,\ldots, U_{k-1},t$. Then \footnotesize
\begin{align}
&\EE[\oF(V_1^k)\ldots \oF(V_{k+1}^k)]\nonumber\\
&=\sum_{n=1}^k\EE[\oF(V^{k-1}_1)\cdots \oF(V^{k-1}_{n-1})\oF(U_k-U^{k-1}_{(n-1)})\oF(U^{k-1}_{(n)}-U_k)\oF(V^{k-1}_{n+1})\cdots \oF(V^{k-1}_k);U^{k-1}_{(n-1)}<U_k<U^{k-1}_{(n)}]\nonumber\\
&= \frac{1}{t}\sum_{n=1}^k\EE\left[\oF(V^{k-1}_1)\cdots \oF(V^{k-1}_{n-1}) \int_{U^{k-1}_{(n-1)}}^{U^{k-1}_{(n)}}\oF(u-U^{k-1}_{(n-1)})\oF(U^{k-1}_{(n)}-u)\dd u\oF(V^{k-1}_{n+1})\cdots \oF(V^{k-1}_k)\right]\nonumber\\
&\leq  \frac{1}{t}\sum_{n=1}^k\EE\left[\oF(V^{k-1}_1)\cdots \oF(V^{k-1}_{n-1})V^{k-1}_n\oF(V^{k-1}_n)\oF(V^{k-1}_{n+1})\cdots \oF(V^{k-1}_k)\right]\nonumber\\
&= \EE\left[\oF(V^{k-1}_1)\cdots \oF(V^{k-1}_k)\right].\label{exp-reset-auxiliary}
\end{align}\normalsize
An inductive argument allows to conclude that the latter is $\leq \oF(t)$. Therefore, by \eqref{exp-reset} and the order-statistics properties of homogeneous Poisson processes, we find that $\TT$ is no bigger under exponential reset. 

Assume now that $\int_0^t\oF(u)\oF(t-u)du=t\oF(t)$ for all $t\in [0,\infty)$. Taking Laplace transforms we find that, with $\hat{\oF}(\lambda):=\int_0^\infty e^{-\lambda t}\oF(t)\dd t$ for $\lambda\in [0,\infty)$, $\hat{\oF}^2+\hat{\oF}'=0$. Solving the differential equation we see that $\hat{\oF}(t)^{-1}-\hat{\oF}(u)^{-1}=t-u$ for all real $t\geq u\geq 0$. Letting $u\downarrow 0$ forces $\hat{\oF}(0)=\TT[\id]<\infty$, and then $\hat{\oF}$ and in turn $\oF$ may be identified, rendering $\TT\sim \Exp(\TT[\id]^{-1})$.
\end{proof}

\subsection{Reset and order in mean}
We turn out attention now to the behavior of $\TT$ under reset in expectation. Assume $m_0:=\TT[\id]<\infty$ and set $m(r):=\TT[\id-r\vert (r,\infty]]$ for $r\in [0,\sup\supp(\TT))$, $m(r):=0$ for $r\in [\sup\supp(\TT),\infty)$. Informally, $m(r)$ may be thought of as the mean residual lifetime of $\TT$ conditionally on it exceeding $r$ (it is only interesting while this can happen with positive probability, viz. for $r\in [0,\sup\supp(\TT))$; the values for the other $r$ are set for convenience). Note $m(0)\geq m_0$ with equality iff $\oF(0)=1$.

In view of its informal meaning, it should come as no surprise to the reader that the mean residual life map, $m$, will feature heavily in the characterizations of this subsection. The following lemma will prove useful, as it expresses $\oF$ in terms of $m$. 

\begin{lemma}\label{lemma}
The map $m:[0,\infty)\to [0,\infty)$ enjoys the following properties: it is locally bounded and right-continuous; it is locally bounded away from zero on $[0,\sup\supp(\TT))$; finally  $$\oF(r)=\frac{m_0}{m(r)}\exp\left(-\int_0^r\frac{\dd v}{m(v)}\right),\quad r\in [0,\sup\supp(\TT)).$$
\end{lemma}
\begin{remark}\label{remark:m}
If, conversely, we take, say, a map $n:[0,\infty)\to (0,\infty)$, continuously differentiable, with $n'\geq -1$, and bounded away from zero, then $[0,\infty)\ni r\mapsto \frac{n(0)}{n(r)}\exp\left(-\int_0^r\frac{\dd v}{n(v)}\right)$ is the tail function of an absolutely continuous law on the Borel sets of $[0,\infty]$, carried by $(0,\infty)$, supported by $[0,\infty]$, and such that $n$ is its mean residual life map.
\end{remark}
\begin{proof}
Only the last part expressing $\oF$ in terms of $m$ is not obvious.  To see why it too is true, note that, for $r\in [0,\sup\supp(\TT))$, we have $m(r)\oF(r)=\int_r^\infty\oF(v)\dd v$. If $\TT$ admits a continuous density  on $[0,\sup\supp(\TT))$ (w.r.t. Lebesgue measure, of course), then this is readily diffferentiated by the fundamental theorem of calculus, and solved for $\oF$. For the general case one may replace $\TT$ with $\TT\star \Gamma(2,\lambda)$ ($\star$ denotes convolution; the fact that  $\Gamma(2,\lambda)=\Exp(\lambda)\star \Exp(\lambda)$ admits a density that is continuous and vanishing at zero ensures that $\TT\star \Gamma(2,\lambda)$ admits a continuous density in turn) and pass to the limit $\lambda\to \infty$ by bounded convergence (grantedly only at all continuity points of $\oF$ from $[0,\sup\supp(\TT))$, but it is enough). 
\end{proof}

As in the previous subsection we identify first the laws that are rendered no smaller (no bigger, invariant) under arbitrary (or just deterministic) reset. Unsurprisingly, the condition is that $m$ is everywhere $\leq m_0$ ($\geq  m_0$,  $=m_0$). Still invariance is the same as being exponentially distributed.

\begin{proposition}\label{proposition:reset-in-mean}
The following statements are equivalent.
\begin{enumerate}[(i)]
\item $\TT$ is no bigger (resp. no smaller, invariant)  under reset in mean: $\TT^\RR[\id]\leq \TT[\id]$ (resp. $\geq \TT[\id]$, $=\TT[\id]$) for all reset laws $\RR$.
\item $\TT$ is no bigger (resp. no smaller, invariant)  under deterministic reset in mean: $\TT^{\delta_r}[\id]\leq \TT[\id]$ (resp. $\geq \TT[\id]$, $=\TT[\id]$) for all $r\in (0,\infty)$.
\item For all $r\in (0,\infty)$: $\TT[\id-r\vert (r,\infty]]\geq \TT[\id]$ (resp. $\leq \TT[\id]$, $=\TT[\id]$) with $\TT((r,\infty])>0$ implicit (resp. whenever $\TT((r,\infty])>0$, with  $\TT((r,\infty])>0$ implicit). 
\end{enumerate}
Furthermore, $m\equiv m_0$ iff  $\TT=\Exp(m_0^{-1})$.
\end{proposition}
\begin{remark}
In order for $\TT$ to be no smaller under reset in mean it is necessary that $\oF(0)=1$. 
\end{remark}
\begin{proof}
Let $\RR$ be a reset law and assume the setting and notation of Definition~\ref{definition:law-reset-by}. Then 
\begin{align}
\TT^\RR[\id]
&=\EE[T_1;T_1\leq R_1]+\sum_{k=2}^\infty (k-1)\EE[R_1;R_1<T_1]\PP(R_1<T_1)^{k-2}\PP(T_1\leq R_1)\nonumber\\
&\hspace{2cm}+\PP(R_1<T_1)^{k-1}\EE[T_1,T_1\leq R_1]\nonumber\\
&=\frac{\EE[T_1;T_1\leq R_1]+\EE[R_1;R_1<T_1]}{1-\PP(T_1>R_1)}
=\frac{\EE[T_1\land R_1]}{1-\PP(T_1>R_1)}.\label{eq:mean-stochastic}
\end{align}
Multiplying both sides by $1-\PP(T_1>R_1)$, we see that  $\TT^\RR[\id]\leq \TT[\id]$ (resp. $\TT^\RR[\id]\geq \TT[\id]$, $\TT^\RR[\id]= \TT[\id]$) for all reset laws $\RR$ iff it holds true for all $\RR$ of the form $\delta_r$, $r\in (0,\infty)$, which in turn is equivalent to, with $\Box$ standing for $\leq$ (resp. for $\geq$, $=$),
\begin{align}
\EE[T_1\land r]&\Box \EE[T_1](1-\PP(T_1>r))\nonumber\\
\Leftrightarrow \EE[T_1]\PP(T_1>r)&\Box \EE[T_1]-\EE[T_1\land r]\nonumber\\
\Leftrightarrow \EE[T_1]\PP(T_1>r)&\Box \EE[T_1\lor r]-r\nonumber\\
\Leftrightarrow \EE[T_1]\PP(T_1>r)&\Box \EE[T_1-r; T_1> r]\nonumber
\end{align}
for all $r\in (0,\infty)$. The final assertion follows from Lemma~\ref{lemma}.
\end{proof}
\begin{remark}\label{remark:checkin}
One may rewrite \eqref{eq:mean-stochastic} as $$\TT^\RR[\id]=\frac{\RR[\TT[\id\land \cdot]]}{\RR[F]}=\frac{\RR[\int_0^\cdot \overline{F}]}{\RR[F]}=\frac{\RR\left[\frac{\int_0^\cdot\overline{F}}{F}F\right]}{\RR[F]}.$$ Unraveling the definitions this agrees with the expression obtained in \cite[Eq.~(6)]{checkin}, and recovered in \cite[Eq.~(E.1)]{Villarroel}, which are given there assuming the existence of densities for $\RR$ and $\TT$. Now, the map $\frac{\int_0^\cdot\overline{F}}{F}:[0,\infty)\to [0,\infty)$ has the limit $m_0$ at $\infty$, $\sup \{\frac{\int_0^t\overline{F}(u)\dd u}{F(t)}:t\in (0,\infty)\}\geq \TT^\RR[\id]\geq \inf \{\frac{\int_0^t\overline{F}(u)\dd u}{F(t)}:t\in (0,\infty)\}$, and therefore plainly $\sup/\inf \TT^\RR[\id]$ over all reset laws $\RR$ is the same as $\sup/\inf_{r\in (0,\infty)}\TT^{\delta_r}[\id]$, and is in fact equal to  $\sup/\inf \{\frac{\int_0^t\overline{F}(u)\dd u}{F(t)}:t\in (0,\infty)\}$ (cf. the comments following \cite[Eq.~(6)]{checkin} \cite[Eq.~(E.1)]{Villarroel}). This yields, once some straightforward computations are made, an alternative proof to the equivalences of Proposition~\ref{proposition:reset-in-mean}.
\end{remark}
To conclude our study of stochastic reset without branching, we would like to express the property of being no bigger (no smaller, invariant) under exponential reset in terms of the residual mean function $m$, since this will best relate it to the findings of Proposition~\ref{proposition:reset-in-mean}. Below we set $\frac{1}{0}:=\infty$, $e^{-\infty}:=0$.
\begin{proposition}
For any given  $\mu\in (0,\infty)$, $\TT^{\Exp(\mu)}[\id]\leq \TT[\id]$ (resp. $\geq \TT[\id]$, $=\TT[\id]$) iff 
\begin{equation}\label{eq:mean-exp}
 \int_0^\infty e^{-\mu t-\int_0^t \frac{\dd u}{m(u)}}\dd t\geq \frac{1}{m_0^{-1}+\mu}\text{ (resp. $\leq \frac{1}{m_0^{-1}+\mu}$, $=\frac{1}{m_0^{-1}+\mu}$)}.
\end{equation}
In particular $\TT$ is invariant under exponential reset in mean (in the obvious meaning of this qualification) iff $\TT=\Exp(m_0^{-1})$.
\end{proposition}

\begin{remark}
In the context of Remark~\ref{remark:m}:
\begin{itemize}
\item As long as $\int_0^t \frac{\dd z}{n(z)}\leq \frac{t}{n(0)}$  for all $t\in (0,\infty)$ and  $n(z)<n(0)$ for some $z\in (0,\infty)$ (which may happen), then  $\TT$ is no bigger under exponential reset in mean, yet $\TT$ is not no bigger under reset in mean.
\item One can take $n$ non-constant and $n\geq n(0)$, in which case $\TT$ is no bigger under reset in mean, but it is not invariant under reset in mean.
\end{itemize}
\end{remark}
\begin{remark}\label{remark:physicisits}
It was essentially observed in \citep[Eq. (4)]{reuveni} (and it also follows from the considerations of the proof below) that  [$\TT^{\Exp(\mu)}[\id]< \TT[\id]$ for all sufficiently small $\mu\in (0,\infty)$] if  $\TT[\id^2]> 2\TT[\id]^2$ and only if $\TT[\id^2]\geq 2\TT[\id]^2$.
\end{remark}
\begin{example}
Let $k\in (0,1)$ and  $\oF(t)=(t+k)^{-2}$ for $t\in  [0,\infty)$. Then $m_0=\TT[\id]=k^{-1}$ and $\TT[\id^2]=\infty$, so the (sufficient) condition of Remark~\ref{remark:physicisits} is trivially met. On the other hand, $m(t)=t+k$ for $t\in [0,\infty)$, and one checks easily that  condition \eqref{eq:mean-exp} fails to hold for all $\mu \in (0,\infty)$, so that $\TT$ is not no bigger under exponential reset in mean. 
\end{example}
\begin{remark}
Beyond ``universally qualifying'' over $\mu\in (0,\infty)$ in \eqref{eq:mean-exp} apparently little (useful) can be said in the direction of characterizing $\TT$ to be no bigger (smaller) under exponential reset in mean.
\end{remark}
\begin{proof}
Set $\RR:=\Exp(\mu)$. Then, in the notation and setting of Definition~\ref{definition:law-reset-by}, $\TT^\RR[\id]\leq \TT[\id]$ is equivalent to (cf. \eqref{eq:mean-stochastic}) $$(1+\mu \EE[T_1])\EE[e^{-\mu T_1}]\geq 1.$$
 Setting $t_0:=\sup\supp(\TT)$, a simple computation, using $\EE[1-e^{-\mu T_1}]=\int_0^\infty \mu e^{-\mu r}\oF(r)\dd r$ and Lemma~\ref{lemma},  reveals that   the latter is further equivalent to
$$
\frac{1}{1+\mu\EE[T_1]}\geq \int_0^{t_0} e^{-\mu r-\int_0^r\frac{\dd v}{m(v)}}\frac{\dd r}{m(r)}.
$$
Since $\int_0^{t_0} e^{-\mu r-\int_0^r\frac{\dd v}{m(v)}}\frac{\dd r}{m(r)}=1-\mu\int_0^{t_0} e^{-\mu r-\int_0^r\frac{\dd v}{m(v)}}\dd r$ we obtain  \eqref{eq:mean-exp}. The claim for the reverse inequality and equality is analogous. The final assertion is then immediate.
\end{proof}

\begin{remark}
We may rewrite \eqref{eq:mean-exp} as $\int_0^\infty (1-\frac{m_0}{m(t)})e^{-\mu t-\int_0^t \frac{\dd u}{m(u)}}\dd t\geq 0$ (resp. $\leq 0$, $=0$). Hence in order for $\TT$ to be no smaller under exponential reset in mean (in the obvious meaning of this qualification), it is necessary that $\oF(0)=1$, at least if $m$ is bounded away from $0$ (we may multiply the indicated inequality by $\mu$ and pass to the limit $\mu\to \infty$).
\end{remark}

It seems plain that being no bigger under reset in mean cannot in general imply the same without the qualification ``in mean''. Still we make an example in which this occurs explicit. 
\begin{example}
Let $\oF(t)=e^{-t-0.1}\mathbbm{1}_{[0,1)}(t)-e^{-t-0.25}\mathbbm{1}_{[1,\infty)}(t)$ for $t\in [0,\infty)$. Using the above characterizations, it is then an elementary exercise to verify that $\TT$ is no bigger under reset in mean, but it is not no bigger under exponential reset.
\end{example}

\section{Restart with branching and stochastic order}\label{section:branching}
Retain the setting of the start of Section~\ref{section:restart}.

\begin{definition}[$l$-fold reset, $l\in \mathbb{N}$]\label{definition:law-reset-l-fold}
Let $\RR$ be a reset law and $l\in \mathbb{N}$. We define a new probability law, $\TT^\RR_l$, on the Borel sets of $[0,\infty]$ as follows. Let $(T^k_i)_{(k,i)\in \mathbb{N}^2,i\leq l^{k-1}}$ be an array of $[0,\infty]$-valued i.i.-with law $\TT$-d. random variables and let $(R_k)_{k\in \mathbb{N}}$ be an independent sequence of $[0,\infty]$-valued i.i.-with law $\RR$-d. random variables, all defined on a common probability space $(\Omega,\FF,\PP)$. Then $\TT^\RR_l$ is the law of the random time $\tilde{T}:\Omega\to [0,\infty]$ specified a.s. by \footnotesize$$\tT=R_1+\cdots+R_{k-1}+T^k_{1}\land \cdots \land T^k_{l^{k-1}}\text{ on }\{R_1<T_1^1,\ldots,R_{k-1}<T_{1}^{k-1} \land\cdots\land T^{k-1}_{l^{k-2}},T^k_{1}\land \cdots \land T^k_{l^{k-1}}\leq R_k\},\quad k\in \mathbb{N}.$$\normalsize
We denote by $\oF^\RR_l$ the tail function of $\TT^\RR_l$. 
\end{definition}

\begin{remark}
 Note, $\TT^\RR_1=\TT^\RR$. 
\end{remark}
In parallel to the results on reset without branching of the previous section we have 
\begin{proposition} \label{proposition:l-fold}
Let $l\in \mathbb{N}_{\geq 2}$. We have the following assertions (in their by now obvious meaning). 
\begin{enumerate}[(1)]
\item\label{l-fold:1} $\TT$ is no bigger under $l$-fold reset iff it is no  bigger under deterministic $l$-fold reset, which occurs iff $\oF(x+y)\geq \oF(x)\oF(y)^l$ whenever $\{x,y\}\subset [0,\infty)$. $\TT$ cannot be invariant under deterministic $l$-fold reset. 
\item\label{l-fold:2}  $\TT$ is no bigger under $l$-fold exponential reset iff $\oF^{\Exp(\mu)}_l(t)\leq \oF(t)$ for all sufficiently small $\mu\in (0,\infty)$ for all $t\in [0,\infty)$, which occurs iff $\frac{1}{t}\int_0^t\oF(u)\oF(t-u)^l\dd u\leq \oF(t)$ for all $t\in (0,\infty)$.
\end{enumerate}
\end{proposition}
\begin{proof}
We may refrain from reporting all the details of the computations, because the ground is already familiar to us from the case of reset without branching. 

\ref{l-fold:1}. Suppose first $\oF(x+y)\geq \oF(x)\oF(y)^l$ whenever $\{x,y\}\subset [0,\infty)$. Then, similarly as in \eqref{eq:stochastic} except that now in the setting and notation of Definition~\ref{definition:law-reset-l-fold}, 
 \begin{align}
\PP(\tT>t)&= \sum_{k=1}^\infty \EE\big[\oF(R_1)\cdots \oF(R_{k-1})^{l^{k-2}}(\oF(t-R_1-\cdots-R_{k-1})^{l^{k-1}}\mathbbm{1}_{\{R_1+\cdots +R_{k-1}\leq t\}}\nonumber\\
&\hspace{5cm}+\mathbbm{1}_{\{R_1+\cdots+ R_{k-1}> t\}}-\oF(R_k)^{l^{k-1}});R_1+\cdots +R_k> t\big]\nonumber\\
&\leq   \sum_{k=1}^\infty \EE\big[\oF(t)\mathbbm{1}_{\{R_1+\cdots +R_{k-1}\leq t\}}+\oF(R_1)\cdots \oF(R_{k-1})^{l^{k-2}}\nonumber\\
&\hspace{4cm}(\mathbbm{1}_{\{R_1+\cdots+ R_{k-1}> t\}}-\oF(R_k)^{l^{k-1}});R_1+\cdots +R_k> t\big]= \oF(t).\label{eq:l-fold}
\end{align}
Conversely, suppose  that $\TT$ is no bigger under $l$-fold deterministic reset. 
Then essentially exactly the same procedure as the one surrounding \eqref{eq:deterministic} shows that indeed $\oF(x+y)\geq \oF(x)\oF(y)^l$ whenever $\{x,y\}\subset [0,\infty)$, $y\leq x$ (hence also if $y>x$), and that furthermore, if $\TT$ is even invariant under $l$-fold deterministic reset, then there is equality in the preceding (albeit only for $y\leq x$). But if the latter prevails, then for all $a\in (0,\infty)$, $n\in \mathbb{N}$, we have $\oF(na)=\oF(a)^{(n-1)l+1}$. In consequence $\oF(\frac{m}{n})=\oF(1)^{\frac{(m-1)l+1}{(n-1)l+1}}$, $\{m,n\}\subset \mathbb{N}$, which cannot be (e.g. plug in $m=1$ \& $n=2$, and then $m=2$ \& $n=4$; note, $l\ne 1$ and $\oF(1)<1$).

\ref{l-fold:2}. Let next $\mu\in (0,\infty)$, set $\RR:=\Exp(\mu)$, and assume the setting of Definition~\ref{definition:law-reset-l-fold}. Define $S_n$, $n\in \mathbb{N}_0$, and $N$ as in the proof of Proposition~\ref{proposition:exponential}. From \eqref{eq:l-fold}, for any given $t\in (0,\infty)$, 
\begin{equation}
\PP(\tT>t)=\sum_{k=0}^\infty \EE\big[\oF(R_1)\cdots\oF(R_k)^{l^{k-1}}\oF(t-S_k)^{l^k}\vert N_t=k]\frac{(\mu t)^k}{k!}e^{-\mu t}\label{exp-reset:l-fold}
\end{equation}
Therefore, a simple modification of the argument of the proof of Proposition~\ref{proposition:exponential} establishes the asserted characterization of when $\TT$ is no bigger under $l$-fold exponential reset.
\end{proof}
\begin{remark}
Let $l\in \mathbb{N}_{\geq 2}$ and let $\RR$ be a reset law. It follows from \eqref{eq:l-fold} by integrating the survival function, that 
\footnotesize
\begin{align*}
\TT^\RR_l[\id]&=\sum_{k=1}^\infty\int_0^\infty \EE\big[\oF(R_1)\cdots \oF(R_{k-1})^{l^{k-2}}\oF(t-R_1-\cdots-R_{k-1})^{l^{k-1}};R_1+\cdots+R_{k-1}\leq t<R_1+\cdots+R_k]\dd t\\
&=\sum_{k=1}^\infty \EE[\oF(R_1)]\cdots \EE[\oF(R_1)^{l^{k-2}}]\EE\left[\int_0^{R_1}\oF(u)^{l^{k-1}}\dd u\right].
\end{align*}\normalsize
If $R_1\sim \Exp(\mu)$ for a $\mu\in (0,\infty)$, then this simplifies to $$\TT^\RR_l[\id]=\mu^{-1}\sum_{k=0}^\infty\prod_{n=0}^k(1-L^{(l^n)}_\TT(\mu)),$$
where for $m\in \mathbb{N}$, $L^{(m)}_\TT$ is the Laplace transform of the minimum of $m$ i.i.-according to the law $\TT$-d. random variables. If $R_1\sim \delta_r$ for an $r\in (0,\infty)$, then it simplifies to $$\TT^\RR_l[\id]=\sum_{k=0}^\infty \oF(r)^{\frac{l^k-1}{l-1}}\int_0^r\oF(u)^{l^{k}}\dd u. $$
But, unlike in the case of reset without branching, neither of these expresssions seems particularly amenable to further analysis in the direction of establishing a ``nice'' explicit condition for when $\TT$ is no bigger/is invariant under (deterministic, exponential) $l$-fold reset in mean. See however \citep[Eq. (5)]{pal} for a condition on when $\TT^{\Exp(\mu)}_l[\id]<\TT[\id]$ for all sufficiently small $\mu\in (0,\infty)$.
\end{remark}

\begin{remark}
Let $l\in \mathbb{N}_{\geq 2}$. An inspection of the proof of item \ref{l-fold:2} above  reveals that  if $\TT$ is to be invariant under $l$-fold exponential reset, then necessarily $\int_0^t\oF(u)^l\oF(t-u)\dd u=t \oF(t)$ for all $t\in (0,\infty)$. In view of Proposition~\ref{proposition:l-fold}~\ref{l-fold:1}, and also the results concerning reset without branching, it seems reasonable to conjecture that this cannot occur, but the author was not able to prove this. (It is immediate that only $\oF$ continuous with $\oF(0)=1$ can/could possibly verify the preceding relation.)
\end{remark}

\section{A parametric example: the Weibull class}\label{section:illustration}
The Weibull class of distributions (on $[0,\infty)$) belongs to the family of extreme value distributions and is one of the simplest and at the same time most widely used classes of lifetime distributions \cite[Section~7.1]{weibull}. It is indexed by only two parameters: the scale, which is inconsequential in our context, as it relates merely to a choice of measurement unit --- we set it equal to unity; and the shape $k\in (0,\infty)$. With this being so, the tail function $\overline{F}_{(k)}$ ($k$ indicating the shape) is given by $\overline{F}_{(k)}(x)=e^{-x^k}$ for $x\in [0,\infty)$; let $\mathcal{T}_{(k)}$ be the associated probability law. 

Simple computations using the conditions derived above reveal that  $\mathcal{T}_{(k)}$ is no bigger (resp. no smaller, invariant) under reset iff $k\leq 1$ (resp. $k\geq 1$, $k=1$), with the equivalence continuing to hold true if ``reset'' is replaced by ``deterministic reset'' or by ``exponential reset'' and/or if the qualification ``in mean'' is added. Furthermore, the condition for $\mathcal{T}_{(k)}$ to be no bigger under (deterministic) $l$-fold reset writes as $\lambda^k+l(1-\lambda)^k\geq 1$ for all $\lambda\in [0,1]$, which is again easily seen to be true iff $k\leq 1$ (consider $\lim_{\lambda\uparrow 1}\frac{1-\lambda^k}{(1-\lambda)^k}$).  Finally, the condition for $\mathcal{T}_{(k)}$ to be no bigger under exponential $l$-fold reset becomes $\int_0^1e^{-t^k(u^k+l(1-u)^k-1)}\dd u\leq 1$ for all $t\in (0,\infty)$, and it is still equivalent to $k\leq 1$ (for $k>1$, $u^k+l(1-u)^k-1<0$ for all $u$ sufficiently close to $1$, no matter how large the $l$, so that letting $t\to\infty$ monotone convergence precludes the stipulated inequality from holding true). 

In summary, within the Weibull class, the condition for being no bigger under reset, possibly only in mean, does not depend at all on whether we consider deterministic, exponential or arbitrary reset, and it also does not depend on the presence of branching, which is certainly not evident a priori. This condition, namely that $k\leq 1$, is even equivalent to the simple second-order condition of Remark~\ref{remark:physicisits} (i.e. to $\Gamma(1+2/k)\geq 2\Gamma(1+1/k)^2$). It is a reflection of the simple structure of the Weibull class: in general, as we have seen, most of the indicated equivalences will fail. 



\bibliographystyle{plainnat}
\bibliography{Stochastic-reset-biblio}

\begin{thebibliography}{15}
\providecommand{\natexlab}[1]{#1}
\providecommand{\url}[1]{\texttt{#1}}
\expandafter\ifx\csname urlstyle\endcsname\relax
  \providecommand{\doi}[1]{doi: #1}\else
  \providecommand{\doi}{doi: \begingroup \urlstyle{rm}\Url}\fi

\bibitem[Barlow et~al.(1996)Barlow, Proschan, and Hunter]{barlow}
R.~E. Barlow, F.~Proschan, and L.~C. Hunter.
\newblock \emph{Mathematical Theory of Reliability}.
\newblock Classics in Applied Mathematics. Society for Industrial and Applied
  Mathematics, 1996.

\bibitem[Chechkin and Sokolov(2018)]{checkin}
A.~Chechkin and I.~M. Sokolov.
\newblock Random search with resetting: A unified renewal approach.
\newblock \emph{Physical Review Letters}, 121:\penalty0 050601, 2018.

\bibitem[Deshpande et~al.(1986)Deshpande, Kochar, and Sing]{desphande}
J.~V. Deshpande, S.~C. Kochar, and H.~Sing.
\newblock Aspects of positive ageing.
\newblock \emph{Journal of Applied Probability}, 23\penalty0 (3):\penalty0
  748--758, 1986.

\bibitem[El-{N}eweihi(1981)]{el}
E.~El-{N}eweihi.
\newblock Stochastic ordering and a class of multivariate new better than used
  distributions.
\newblock \emph{Communications in Statistics - Theory and Methods}, 10\penalty0
  (16):\penalty0 1655--1672, 1981.

\bibitem[Eliazar(2017)]{Eliazar}
I.~Eliazar.
\newblock Branching search.
\newblock \emph{{EPL} (Europhysics Letters)}, 120\penalty0 (6):\penalty0 60008,
  2017.

\bibitem[Lapeyre~Jr. and Dentz(2019)]{unified}
G.~J. Lapeyre~Jr. and M.~Dentz.
\newblock Unified approach to reset processes and application to coupling
  between process and reset.
\newblock 2019.
\newblock arXiv:1903.08055v3.

\bibitem[Nofal(2012)]{nofal}
Z.~M. Nofal.
\newblock On the class of new better than used of life distributions.
\newblock \emph{Applied Mathematical Sciences}, 6\penalty0 (137):\penalty0
  6809--6817, 2012.

\bibitem[Pal and Reuveni(2017)]{pal-reuveni}
A.~Pal and S.~Reuveni.
\newblock First passage under restart.
\newblock \emph{Physical Review Letters}, 118:\penalty0 030603, 2017.

\bibitem[Pal et~al.(2019)Pal, Eliazar, and Reuveni]{pal}
A.~Pal, I.~Eliazar, and S.~Reuveni.
\newblock First passage under restart with branching.
\newblock \emph{Physical Review Letters}, 122:\penalty0 020602, 2019.

\bibitem[Rao and Damaraju(1992)]{rao}
B.~R. Rao and C.~V. Damaraju.
\newblock New better than used and other concepts for a class of life
  distributions.
\newblock \emph{Biometrical Journal}, 34\penalty0 (8):\penalty0 919--935, 1992.

\bibitem[Reuveni(2016)]{reuveni}
S.~Reuveni.
\newblock Optimal stochastic restart renders fluctuations in first passage
  times universal.
\newblock \emph{Physical Review Letters}, 116\penalty0 (17):\penalty0 170601,
  2016.

\bibitem[Rinne(2008)]{weibull}
H.~Rinne.
\newblock \emph{The Weibull Distribution: A Handbook}.
\newblock CRC Press, 2008.

\bibitem[Shaked and Shanthikumar(2007)]{reset}
M.~Shaked and J.~G. Shanthikumar.
\newblock \emph{Stochastic Orders}.
\newblock Springer Series in Statistics. Springer New York, 2007.

\bibitem[Villarroel and Montero(2018)]{Villarroel}
J.~Villarroel and M.~Montero.
\newblock Continuous-time ballistic process with random resets.
\newblock \emph{Journal of Statistical Mechanics: Theory and Experiment},
  2018\penalty0 (12):\penalty0 123204, 2018.

\bibitem[Weiss(1956)]{weiss}
G.~H. Weiss.
\newblock On the theory of replacement of machinery with a random failure time.
\newblock \emph{Naval Research Logistics Quarterly}, 3\penalty0 (4):\penalty0
  279--293, 1956.

\end{thebibliography}
\end{document}